\documentclass[a4paper, 12pt]{amsart}
\usepackage[utf8]{inputenc}

\usepackage{hyperref}
\usepackage[margin=1in]{geometry}
\usepackage{amsfonts, amsmath, mathtools, amsthm, amssymb, mathrsfs}
\usepackage{graphicx, caption, subcaption, multirow, wrapfig}

\DeclareMathOperator{\ric}{Ric}
\DeclareMathOperator{\tr}{tr}

\theoremstyle{plain}
\newtheorem{thm}{Theorem}[section]
\newtheorem{pro}[thm]{Proposition}
\newtheorem{lem}[thm]{Lemma}

\title{Doubly warped product Einstein metrics on spheres}
\author{Qiu Shi Wang}
\address{Mathematical Institute\\ University of Oxford \\
	Oxford\\
	OX2 6GG\\
	United Kingdom} \email{wangqs@maths.ox.ac.uk, qiu.s.wang@mail.mcgill.ca}
\date{9 October 2026}
\thanks{The author is supported by NSERC/CRSNG [award number CGS D - 598837 - 2025] and the Engineering and Physical Sciences Research Council [grant number EP/W524311/1].}

\begin{document}
	
\begin{abstract}
 	We present a simple computer-assisted procedure to construct $SO(d_1+1)\times SO(d_2+1)$-invariant cohomogeneity one Einstein metrics, and use it to recover known Einstein metrics on $S^{10}$ and $S^{12}$, as well as find new ones on $S^{11}$, $S^{12}$, $S^{13}$ and $S^7\times S^3$.
\end{abstract}
\maketitle

\section{Introduction}

A Riemannian manifold $(M, g)$ is said to be an Einstein manifold if $\ric g = \Lambda g$ for some $\Lambda\in\mathbb{R}$. The unit-radius round metric on the sphere $S^{n+1}$ is Einstein with constant $\Lambda = n$.

The study of non-round Einstein metrics on spheres began in the homogeneous setting, in which a full classification was obtained by Ziller \cite{Z82}. Later, many examples of non-round Sasaki--Einstein metrics on spheres were discovered (e.g. in \cite{BGK05, GK07}), with a recent work of Liu--Sano--Tasin \cite{LST25} proving the existence of infinitely many families of Sasaki--Einstein metrics on $S^{2m+1}$ for all $m\geq 2$. All the above-mentioned metrics are in odd dimensions; far fewer examples are known in even dimensions. Assuming invariance under the usual cohomogeneity one action of $SO(d_1+1)\times SO(d_2+1)$, Böhm \cite{B98} discovered infinitely many non-round Einstein metrics on $S^m$ for $m\in\{5, 6, 7, 8, 9\}$. An additional example, which is nearly Kähler, was found on $S^6$ by Foscolo--Haskins \cite{FH17}. More recently, three more $SO(d_1+1)\times SO(d_2+1)$-invariant metrics were found on $S^{10}$ by Nienhaus--Wink \cite{NW25}, and Buttsworth--Hodgkinson used a novel, computationally-assisted procedure to construct another on $S^{12}$ \cite{BH26}.

In the present work, we draw from and build upon \cite{BH26} and \cite{NW25} to develop a simple and flexible computer-assisted procedure for constructing $SO(d_1+1)\times SO(d_2+1)$-invariant Einstein metrics on either $S^{d_1+d_2+1}$ or $S^{d_1}\times S^{d_2+1}$. Our methods apply in any dimension, but we will restrict ourselves to dimensions 10 and above ($d_1+d_2\geq 9$), as reconstructing any single one of Böhm's infinite families of metrics in dimensions 5 through 9 would be of very limited interest. Our main result is

\begin{thm}\label{maintheorem}
    There exist non-standard $SO(d_1+1)\times SO(d_2+1)$-invariant Einstein metrics on $S^{d_1+d_2+1}$ or $S^{d_1}\times S^{d_2+1}$ as follows. All metrics mentioned are pairwise distinct.
    \begin{itemize}
        \item A non-round metric on $S^{10}$ in each of the cases $(d_1, d_2) = (2, 7), (3, 6), (4, 5)$, and a metric on $S^7\times S^3$ (so $(d_1, d_2) = (7, 2)$) which is not the product of round metrics.
        \item A non-round metric on $S^{11}$ in each of the cases $(d_1, d_2) = (2, 8), (3, 7)$.
        \item Two non-round metrics on $S^{12}$, both with $(d_1, d_2) = (2, 9)$.
        \item Two non-round metrics on $S^{13}$, both with $(d_1, d_2) = (2, 10)$.
    \end{itemize}
\end{thm}

The three metrics on $S^{10}$ are those found in \cite{NW25} and one of the two metrics on $S^{12}$ is the one in \cite{BH26}. All other metrics appear to be new; in fact, the metric on $S^7\times S^3$ and the two metrics on $S^{11}$ were conjectured to exist in \cite{NW25}, and the second metric on $S^{12}$ was conjectured to exist in \cite{BH26}.

Consistently with Conjecture B of \cite{NW25}, our numerical experiments suggest that there are no other (non-round and non-product) $SO(d_1+1)\times SO(d_2+1)$-invariant positive Einstein manifolds in dimensions $\geq 10$.

We now introduce the setup and notation to be used. If $d_1=1$ or $d_2=1$, then the only invariant Einstein metrics are (products of) round metrics \cite[Remark 5.4]{NW25}. Hence, we assume that $d_1 ,d_2\geq 2$, and consider $SO(d_1+1)\times SO(d_2+1)$-invariant cohomogeneity one metrics on $(0, T)\times (S^{d_1}\times S^{d_2})$ of the form 
\begin{equation}\label{metric}
    dt^2 + f_1^2(t)g_{S^{d_1}} + f_2^2(t)g_{S^{d_2}},
\end{equation}
where $f_1, f_2\geq 0$ and $g_{S^{d_i}}$ is the unit-radius round metric on $S^{d_i}$. On a cohomogeneity one manifold with at least one singular orbit, the Einstein equation $\ric g = \Lambda g$ is given in terms of the shape operator $L$ and the Ricci endomorphism $r$ of the principal orbits by \cite[Corollary 2.6]{EW00}
\begin{equation*}
    -\dot L - \tr(L)L + r = \Lambda\,\mathrm{id}, \qquad -\tr(\dot L) - \tr(L^2) = \Lambda,
\end{equation*}
where the dot denotes $\frac{d}{dt}$. For the metric (\ref{metric}), $L = (\frac{\dot f_1}{f_1}\mathrm{id}_{d_1}, \frac{\dot f_2}{f_2}\mathrm{id}_{d_2})$ and $r=(\frac{d_1-1}{f_1^2}\mathrm{id}_{d_1}, \frac{d_2-1}{f_2^2}\mathrm{id}_{d_2})$. From this point on, we will fix the Einstein constant $\Lambda$ to be $n\equiv d_1+d_2$.

In order to form a smooth metric on the sphere, one must complete $(0,T)\times (S^{d_1}\times S^{d_2})$ with two singular orbits, with a smooth collapse of $S^{d_1}$ at one (say at $t=T$) and a smooth collapse of $S^{d_2}$ at the other (say at $t=0$). The functions $f_1, f_2$ would then need to satisfy the boundary conditions
\begin{gather*}
    f_1(0) = \alpha^{-1}\sqrt{d_1-1}, \quad \dot f_1(0) = 0,\quad f_2(0) = 0, \quad \dot f_2(0) = 1\\
    f_1(T) = 0, \quad \dot f_1(T) = -1, \quad f_2(T) = \omega^{-1}\sqrt{d_2-1}, \quad \dot f_2(T) = 0
\end{gather*}
for some $\alpha, \omega>0$. The authors of \cite{BH26} consider the Einstein ODE as two shooting problems, one from each singular orbit. To prove that the two solutions match up in the middle, they perform a careful analysis of the linearisation of the system with respect to $\alpha$ and $\omega$, a delicate and involved process. Our approach will bypass this process altogether, and along with it most of the technicalities in \cite{BH26}.

Instead of constructing a smooth metric directly, we will consider a pair of singular solutions with initial data $\alpha_1$, $\alpha_2$ corresponding to two different types of singular behaviour, one with $f_1\to\infty, f_2\to 0$, and the other with $f_1\to 0, f_2\to\infty$. As developing one of the above singularities is an open condition in the initial data, by an intermediate value-type argument, there must be at least one solution in between $\alpha_1$ and $\alpha_2$ which exhibits different endtime behaviour. We show in Lemma \ref{whichcompletion2}, by counting the number of critical points of $w \equiv f_1/f_2$ and making use of properties of the Einstein ODE described in \cite{NW25}, that under certain conditions the above-mentioned solution must correspond to a metric on the sphere.

It suffices therefore to prove the existence of these two singular solutions with singularity types $f_1\to\infty$, $f_2\to 0$ and $f_1\to 0$, $f_2\to \infty$. To do so, we use a computer-assisted approach in two steps analogous to that used in \cite{W26} to construct noncompact Einstein 4-manifolds. Firstly, we construct incomplete solutions which ``approach'' the desired singular behaviour at their respective end times. This is done exactly as in \cite{BH26} (slightly generalised to all dimensions $(d_1, d_2)$); in fact, we only need to use up to Lemma 5 (with $k=0$) of \cite{BH26}. For further details, we refer the reader to the above-mentioned paper; in short, they follow the below blueprint:
\begin{enumerate}
    \item Produce a high-precision heuristic approximate solution of the Einstein ODE using an arbitrary-precision numerical power series solver.
    \item Use Chebyshev interpolation to obtain from the above heuristic solution a smooth approximate solution, and use interval arithmetic to obtain a rigorous upper bound on the \textit{a posteriori} ODE error.
    \item With good analytic control on the linear part of the ODE, use the Schauder fixed-point theorem to deduce the existence of a solution which must be close to the approximate solution.
\end{enumerate}
In \cite{BH26}, solutions are only constructed up to the principal orbit with maximal volume (i.e. ``roughly halfway''), whereas we must extend our solutions to a time approaching a singularity, which is much more computationally intensive. To drastically decrease the precision required and render the problem feasible within a reasonable amount of processing time, we sharpen in Proposition \ref{solutionoperatorproposition} some estimates used in key steps of \cite{BH26}.

We also show in Lemma \ref{boxeslemma} that the solutions constructed in the first step reach regions of the dynamical system (\ref{ZDeltaHequation}) in which they must fall into sinks corresponding to the desired singular behaviour.

Finally, we apply this process with the specific values listed in Table \ref{maintable} to produce several Einstein metrics detailed in Theorem \ref{maintheorem}, many of which appear to be new.

The paper is organised as follows. First, we study in \S\ref{intermediatesection} the Einstein ODE using the setup of \cite{NW25} to reduce the problem to finding two singular metrics with specific endtime behaviours. Then, in \S\ref{computationalsection}, we summarise and slightly refine the Buttsworth--Hodgkinson procedure to produce (incomplete) solutions, and count the critical points of $w$. This enables us to construct pairs of singular metrics in \S\ref{theoremsection} in order to prove Theorem \ref{maintheorem}.

\textit{Remarks.} Besides the 1-dimensional nature of the space of initial data, our method of proof makes little use of specific features of the geometric situation at hand. As a result, one could conceivably apply it to other cohomogeneity one geometric initial value problems with one parameter, e.g. the Ricci soliton condition in certain setups.

An independent, non-computer-assisted proof of the existence of the non-product metric on $S^7\times S^3$ appears in \cite{H26}.

\section{Critical points of $w$ and intermediate value arguments}\label{intermediatesection}

In this section, we build upon the setup of \cite{NW25} to show that smooth completions exist ``in between'' each pair of singular solutions of opposite singularity type in parameter space. To do so, we first introduce some notation and results from \cite{NW25}, then use a simple continuity argument to prove the main result of the section, Lemma \ref{whichcompletion2}, which says that the topology of the smooth completion near the second singular orbit, i.e. whether $S^{d_1}$ or $S^{d_2}$ collapses, depends on the number of critical points of $w=f_1/f_2$ for the singular solutions reaching the sinks.

\subsection{Definitions and known results}
Consider the warping functions $f_1(t), f_2(t)$ in (\ref{metric}). Let $\mathscr{L} = \frac{1}{\sqrt{\tr (L)^2 + n^2}}$ and define as in \cite{C24} a new coordinate $s$ by $\frac{d}{ds} = \mathscr{L}\frac{d}{dt}$, with $s=0$ at the unique $t$ at which $\tr(L)=0$. Denote $d/dt$ by a dot and $d/ds$ by a prime, and consider the new variables, for $i=1,2$,
\begin{equation}\label{NWcoordinates}
   \begin{gathered}
        X_i = \mathscr{L}\frac{\dot f_i}{f_i}, \quad Y_i = \mathscr{L}\frac{1}{f_i}, \quad H = \mathscr{L}\tr(L)\\
        \Delta = X_1 - X_2, \qquad Z = (d_1-1)Y_1^2 - (d_2-1)Y_2^2.
    \end{gathered} 
\end{equation}
\begin{pro}[Nienhaus--Wink {\cite[\S 2.3]{NW25}}]
    Let $\mathcal{S}\subset\{(Z, \Delta, H)\in \mathbb{R}^3\}$ be the compact set given by
    \begin{equation*}
        \mathcal{S} = \left\{d_1Z + \frac{d_1d_2}{n}\Delta^2\leq \frac{n-1}{n}\right\}\cap \left\{-d_2Z + \frac{d_1d_2}{n}\Delta^2 \leq \frac{n-1}{n}\right\} \cap \{H^2\leq 1\}.
    \end{equation*}
    Einstein metrics of the form (\ref{metric}) correspond to solutions within $\mathcal{S}$ of the system
    \begin{equation}\label{ZDeltaHequation}
        \begin{aligned}
            Z' &= \frac{2}{n}\Delta\left(d_1d_2Z\Delta H + \frac{d_1d_2}{n}\Delta^2 - \frac{n-1}{n} + (d_1 - d_2)Z\right)\\
            \Delta' &= \Delta H\left(\frac{d_1d_2}{n}\Delta^2 - \frac{n-1}{n}\right) + Z\\
            H' &= -\frac{1-H^2}{n}(d_1d_2\Delta^2 + 1).
        \end{aligned}
    \end{equation}
    The set $\mathcal{S}$ and its boundary are preserved by (\ref{ZDeltaHequation}), and solutions in $\mathcal{S}$ exist for all times $s$. The coordinate $H$ on $[-1,1]$ is non-increasing, in fact strictly decreasing for $H\neq 1, -1$.
\end{pro}
\begin{pro}[Props. 2.5 and 2.6 of \cite{NW25}]
    The fixed points of (\ref{ZDeltaHequation}) within $\mathcal{S}$ are given in $(Z, \Delta, H)$ coordinates by
    \begin{itemize}
        \item $p_1^\pm = (\frac{d_1-1}{d_1^2}, \pm \frac{1}{d_1}, \pm 1)$, corresponding to a singular orbit formed by the smooth collapse of $S^{d_1}$.
        \item $p_2^\pm = (-\frac{d_2 - 1}{d_2^2}, \mp \frac{1}{d_2}, \pm 1)$, corresponding to a smooth collapse of $S^{d_2}$.
        \item $\mathrm{cone}^\pm = (0, 0, \pm 1)$, corresponding to a collapse of both factors.
        \item $q_1^\pm = (0, \mp \sqrt{\frac{n-1}{d_1d_2}}, \pm 1)$ and $q_2^\pm = (0, \pm \sqrt{\frac{n-1}{d_1d_2}}, \pm 1)$, corresponding to singular solutions with $f_1\to 0, f_2\to\infty$ and $f_1\to\infty, f_2\to 0$ respectively.
    \end{itemize}
    They are all hyperbolic: $q_i^+$ are sources, $q_i^-$ are sinks and $p_i^\pm$, $\mathrm{cone}^\pm$ are saddles. 
    
    Furthermore, the stable manifolds of $p_i^-$ are $2$-dimensional and intersect $\partial\mathcal{S}$ transversally near $p_i^-$, and $\mathrm{cone}^-$ is a source when (\ref{ZDeltaHequation}) is restricted to $\mathcal{S}\cap \{H = -1\}$.
\end{pro}
The heteroclinic trajectory $\{(0,0,H): H\in(-1,1)\}$ going from $\mathrm{cone}^+$ to $\mathrm{cone}^-$ corresponds to the stable manifold of $\mathrm{cone}^-$ and results in the ``cone solution'', the sine suspension over the principal orbit. 

Solutions which are ``close enough'' to a sink $q_i^-$ must fall into it; this notion is quantified in the following lemma.
\begin{lem}\label{boxeslemma}
	For $i=1, 2$, a solution $(Z, \Delta, H)$ of (\ref{ZDeltaHequation}) within $\mathcal{S}$ satisfying $H<0$ and $(-1)^{i+1}\Delta H + 1/d_{3-i}<0$ at any point in time must converge to $q_i^-$.
\end{lem}
\begin{proof}
	First, notice that $\varepsilon_\Delta \equiv \frac{n-1}{n} - \frac{d_1d_2}{n}\Delta^2\geq 0$ within $\mathcal{S}$, with equality exactly on the heteroclines going from $q_2^+$ to $q_1^-$ along $(Z, \Delta) = (0, \sqrt{\frac{n-1}{d_1d_2}})$ or from $q_1^+$ to $q_2^-$ along $(Z, \Delta) = (0, -\sqrt{\frac{n-1}{d_1d_2}})$. As $\Delta' = -\Delta H \varepsilon_\Delta + Z$ and $-\varepsilon_\Delta/d_2 \leq Z \leq \varepsilon_\Delta/d_1$, we have that $-(\Delta H + d_2^{-1})\varepsilon_\Delta\leq \Delta' \leq (-\Delta H + d_1^{-1})\varepsilon_\Delta$.
	
	We claim that if a solution reaches a point $p^*=(Z^*, \Delta^*, H^*)$, not on the above-mentioned heteroclines, with $-(\Delta^*H^*+d_2^{-1})>0$ and $H^*<0$, then $\Delta'>0$ for all future times. To see this, suppose otherwise, and consider the first point $\tilde p = (\tilde Z, \tilde \Delta, \tilde H)$ along the integral curve starting from $p^*$ at which $\Delta'(\tilde p)=0$. We have $\tilde H \leq H^*<0$, and since $\Delta'(p^*)>0$, we also have $\tilde \Delta > \Delta^* > 0$. However,
	\begin{equation*}
		\Delta'(\tilde p)/\varepsilon_{\tilde\Delta} \geq -(\tilde \Delta\tilde H + d_2^{-1})\geq -(\Delta^*H^* + d_2^{-1})>0,
	\end{equation*}
	which is a contradiction.
	
	Consequently, such a solution remains within the region $R_1\equiv \{\Delta >-1/(d_2H^*)\} \cap \{H\leq H^*\}\cap \mathcal{S}$. As the only fixed point of (\ref{ZDeltaHequation}) within $R_1$ is $q_1^-$, Proposition 2.15 of \cite{NW25} implies that the solution must converge to $q_1^-$.
	
	Similarly, if $-\Delta^*H^* + d_1^{-1}<0$, then $\Delta'<0$ for all future times, so the same conclusion holds for $R_2\equiv \{\Delta<1/(d_1H^*)\}\cap \{H\leq H^*\} \cap \mathcal{S}$, which contains $q_2^-$.
\end{proof}
\begin{wrapfigure}{r}{0.3\textwidth}
	\includegraphics[width=\linewidth]{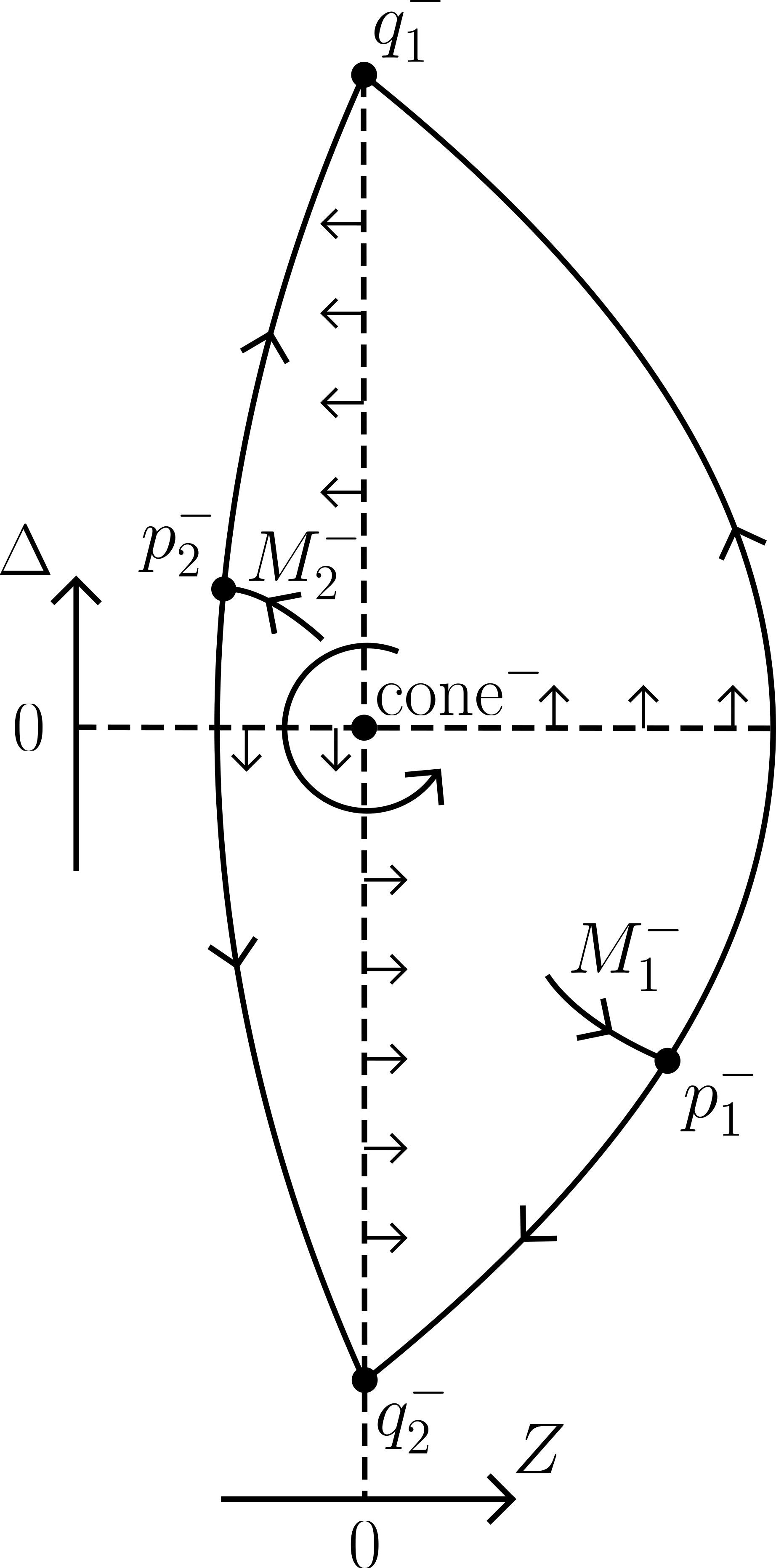}
	\caption{The slice $\mathcal{S}\cap \{H=-1\}$ in the $(Z, \Delta)$-plane, with fixed points and heteroclinic orbits along the boundary. \vspace{-.5in}}
	\label{H=-1slice}
\end{wrapfigure}
Furthermore, all solutions exhibit a rotational behaviour:
\begin{pro}[Prop. 2.13 of \cite{NW25}]
    Let $(Z, \Delta, H)$ be a solution of (\ref{ZDeltaHequation}) within $\mathcal{S}$ with $(Z, \Delta)\neq (0,0)$. If it enters a quadrant in the $(Z, \Delta)$-plane, it either remains there for all times, or exits it into the next quadrant going counterclockwise around the cone solution.
\end{pro}
Finally, let $w(t) = f_1(t)/f_2(t)$; its critical points are non-degenerate \cite[Lemma 4.2]{B98} and occur precisely when $\Delta = 0$. Henceforth, the term ``critical points'' refers to those of $w$.

\subsection{Continuity argument}
We will assume that solutions begin from $p_2^+$, so that the initial conditions are $f_2(0)=0$, $f_1(0)>0$. Denote by $\eta_x$ the solution with $f_1(0) = x$, and consider the behaviour of solutions as $s\to\infty$ and $H\to -1$.
\begin{lem}\label{whichcompletion1}
    Suppose that $\eta_x$ converges to the sink $q_i^-$ with $k$ critical points of $w$, and that there exists an $a>0$ such that $\eta_y$ also converges to $q_i^-$ for all $y\in (a, x]$ or for all $y\in  [x, a)$, but $\eta_a$ does not.
    \begin{itemize}
        \item If $i=1$, then $k\geq 1$, and either $\eta_a$ goes to $p_2^-$ with $k$ critical points, or to $p_1^-$ with $k-1$ critical points.
        \item If $i=2$, then $\eta_a$ either goes to $p_1^-$ with $k$ critical points, or to $p_2^-$ with $k-1$ critical points (which cannot occur if $k=0$).
    \end{itemize}
\end{lem}
\begin{proof}
    Let $I=(a, x]$ or $[x,a)$. First, we note that $\eta_y$ has $k$ critical points for each $y\in I$, as convergence to a sink is an open condition in the initial data.
    
    The hyperbolic fixed points $p_i^-$ have a two-dimensional stable manifold $M_i^-$ and an unstable manifold which coincides with the heteroclinic orbits from $p_i^-$ to the sinks $q_j^-$. By the Hartman--Grobman theorem, there exists an $\varepsilon >0$ such that trajectories beginning from $B(p_i^-, \varepsilon)\cap \mathcal{S}$ either lie on $M_i^-$, or go to one of the sinks $q_j^-$ following a path near the heteroclinic orbit from $p_i^-$ to $q_j^-$. More specifically, given the layout of the fixed points $q_l^-$, $p_l^-$ in the $(Z, \Delta)$-plane (see Figure \ref{H=-1slice}), we have that
    \begin{itemize}
        \item For $i=1$, the integral curve either lies on $M_1^-$, goes from the ball to $q_2^-$, or crosses $\{\Delta = 0\}$ and goes to $q_1^-$.
        \item For $i=2$, the integral curve either lies on $M_2^-$, goes from the ball to $q_1^-$, or crosses $\{\Delta=0\}$ and goes to $q_2^-$.
    \end{itemize}
    
    Consider the $\omega$-limit set $\omega(p)$ of a point $p = (Z^*, \Delta^*, H^*)\in \mathcal{S}$ with $H^*\neq 1$ and $(Z^*, \Delta^*)\neq (0, 0)$. $\omega(p)$ is compact, connected, non-empty and invariant under the flow of the ODE. As $H'\leq 0$ with equality if and only if $H=\pm 1$, $\omega(p)$ lies on $\mathcal{S}\cap \{H=-1\}$. As $p$ does not lie on the stable manifold $\{Z=\Delta=0\}$ of $\mathrm{cone}^-$, by the Hartman--Grobman theorem, there is a neighbourhood of $\mathrm{cone}^-$ disjoint from $\omega(p)$. Since every trajectory in $\mathrm{int}(\mathcal{S}\cap \{H=-1\})$ emanates from $\mathrm{cone}^-$ \cite[Proposition 2.15]{NW25}, we have that the invariant set $\omega(p)$ must lie on the boundary of $\mathcal{S}\cap \{H=-1\}$. If $\omega(p)$ contained any non-fixed point, then it must also contain an entire heteroclinic orbit going from some $p_i^-$ to some $q_j^-$, which is impossible as the $q_j^-$'s are sinks. As a result, $\omega(p)$ is the singleton set containing $p_1^-$, $p_2^-$, $q_1^-$ or $q_2^-$, to which the solution must converge since it remains within the compact set $\mathcal{S}$.

    The solution $\eta_a$ emanates from $p_2^+$ and thus is not the cone solution. By definition, it does not converge to $q_i^-$; in fact, as converging to a sink is an open condition in the initial data, it cannot converge to either $q_1^-$ or $q_2^-$. By the above considerations, the only remaining possibilities are that $\eta_a$ converges to either $p_1^-$ or $p_2^-$. Consequently, there exists a $\lambda>0$ such that for each $y\in I$ with $|y-a|<\lambda$, $\eta_y$ reaches $B(p_1^-, \varepsilon)$ or $B(p_2^-, \varepsilon)$ with the same number of critical points of $w$, equivalently zeroes of $\Delta$, as $\eta_a$. By definition $\eta_y$ does not reach the stable manifold $M_i^-$, so the lemma follows.
\end{proof}

\begin{lem}\label{whichcompletion2}
    Suppose that there exist $x, y\in(0,\infty)$ such that either of the following holds for some $l\in\mathbb{N}$:
    \begin{enumerate}
        \item $\eta_{x}$ reaches $q_1^-$ with $l$ critical points and $\eta_{y}$ reaches $q_2^-$ with $l+1$ critical points.
        \item $\eta_x$ reaches $q_1^-$ with $l$ critical points and $\eta_y$ reaches $q_2^-$ with $l-1$ critical points.
    \end{enumerate}
    Then, we have respectively that there exists a $z$ between $x$ and $y$ such that
    \begin{enumerate}
        \item $\eta_z$ goes to the zero $p_2^-$.
        \item $\eta_z$ goes to the zero $p_1^-$.
    \end{enumerate}
\end{lem}
\begin{proof}
    We have seen that it is possible, by continuous variation of $z$, for $\eta_z$ to move from the $q_i^-$'s to the $p_i^-$'s, with a loss of 0 or 1 critical point. It is impossible to move from $q_1^-$ to $q_2^-$, as they are both sinks. It is also impossible to move continuously from $p_1^-$ to $p_2^-$, as the respective stable manifolds within the cylinder, $M_1^-\cap \mathcal{S}$ and $M_2^-\cap \mathcal{S}$, are closed and disjoint, and therefore some positive distance apart from one another.

    Therefore the only way to lose a critical point, which is required to go from $y$ to $x$, is to go to the correct zero, as per Lemma \ref{whichcompletion1}.
\end{proof}

\textit{Remark.} For solutions starting from $p_2^+$, the integer $l$ is always odd.

\section{Computer-assisted construction of singular solutions}\label{computationalsection}

In this section, we develop a procedure to construct solutions which will satisfy the hypotheses of Lemma \ref{whichcompletion2}. To do so, we make use of the computationally-assisted methods introduced in \cite{BH26} to construct an incomplete Einstein metric up to some time $t_f>0$, such that the endpoint of the solution is close enough to a sink $q_i^-$, in the sense of Lemma \ref{boxeslemma}, for it to ``fall in'' and converge to it. We also count the critical points of $w$.

\subsection{Constructing the solution}
Recall from \cite{BH26} that the Einstein equation starting from $p_2^+$ with $f_1(0) = \alpha^{-1}\sqrt{d_1-1}$, $\alpha>0$, is equivalent to the singular i.v.p. 
\begin{equation}\label{BHivp}
    \dot\eta(t) = \frac{1}{t}L_{d_1d_2}\eta(t) + B_{d_1d_2}(\eta(t), \eta(t)), \qquad \eta(0) = (0, 0, 0, \alpha, \sqrt{n})^T,
\end{equation}
where $L_{d_1d_2} = \begin{pmatrix*} 0 & 0 & 0 & 0 & 0\\
	0 & -d_2 & 0 & 0 & 0 \\
	0 & 2 & -2 & 0 & 0 \\
	0 & 0 & 0 & 0 & 0 \\
	0 & 0 & 0 & 0 & 0
\end{pmatrix*}$, $B_{d_1d_2}$ is the symmetric bilinear form given by
\begin{equation*}
	B_{d_1d_2}(x,y) = \begin{pmatrix*} -\frac{1}{2d_1}(x_2y_4 + x_4y_2 + x_1y_2 + x_2y_1)\\
	-\frac{1}{2}(x_2y_3 + x_3y_2) + d_1(x_1y_1 + x_4y_4 + x_1y_4 + x_4y_1) - d_1x_5y_5\\
	-\frac{1}{d_2}x_3y_3 + \frac{1}{d_2}(x_2y_3 + x_3y_2) - (\frac{1}{d_1} + \frac{1}{d_2})x_2y_2 - x_5y_5\\
	0\\
	0
	\end{pmatrix*},
\end{equation*}
 and $\eta = (\eta_1, \eta_2, \eta_3, \eta_4, \eta_5)^T$ is related to the warping functions $f_1, f_2$ of (\ref{metric}) by
\begin{equation}\label{BHcoordinates}
    \eta_1 = \frac{\sqrt{d_1-1}}{f_1}-\alpha, \quad \eta_2 = \frac{d_1\dot f_1}{f_1}, \quad \eta_3 = \frac{d_1\dot f_1}{f_1} + \frac{d_2\dot f_2}{f_2} - \frac{d_2}{t}, \quad \eta_4 = \alpha,\quad \eta_5 = \sqrt{n}.
\end{equation}
Note that $\eta_4, \eta_5$ are constant functions which are included in the system for notational convenience. A procedure to construct solutions to (\ref{BHivp}) is already detailed in \cite{BH26}, so we only explain how our work differs from the former.

First, let $\hat\eta(t)$ be a smooth approximate solution to (\ref{BHivp}) such that $\hat\eta(0) = \eta(0)$ and let $\mu = \eta - \hat \eta$. Write $\mu(t) \cdot B_{d_1d_2}(\hat\eta(t), \mu(t))$ as $\mu^t M_l(\hat\eta(t)) \mu$, where $M_l(\eta)$ is given by the symmetrization of
\begin{equation*} 
    \begin{pmatrix}
        -\frac{\eta_2}{2d_1} & (d_1 - \frac{1}{2d_1})(\eta_1 + \eta_4) & 0 & -\frac{\eta_2}{2d_1} & 0\\
        0 & -\frac{\eta_3}{2} & -(\frac{1}{2} + \frac{1}{d_1} + \frac{1}{d_2})\eta_2 + \frac{1}{d_2}\eta_3 & d_1(\eta_1 + \eta_4) & -d_1 \eta_5\\
        0 & 0 & -\frac{1}{d_2}\eta_3 + \frac{1}{d_2} \eta_2 & 0 & -\eta_5\\
        0 & 0 & 0 & 0 & 0\\
        0 & 0 & 0 & 0 & 0
    \end{pmatrix}.
\end{equation*}
Fix some time $\hat t$, and use Taylor's theorem for Banach spaces to expand
\begin{equation*}
    M_l(\hat\eta(t)) = M_l(\hat\eta(\hat t)) + \dot M_l(\hat\eta(\hat t))(t-\hat t) + \frac{1}{2}M^{(2)}(t)(t-\hat t)^2.
\end{equation*}
At each time $t$, we then have the following estimate of the operator norm of $M_l$:
\begin{equation*}
    \|M_l(\hat\eta(t)\|\leq \|M_l(\hat\eta(\hat t))\| + |t-\hat t| \|\dot M_l(\hat\eta(\hat t))\| + \frac{1}{2}(t-\hat t)^2C_2,
\end{equation*}
where $C_2>0$ is chosen such that $\|M^{(2)}(t)\|\leq \sup_{t\in(0, t_f)} \|\ddot{M}_l(\hat\eta(t))\|\leq C_2$. To sharpen this bound, we repeat the process with a grid of times $0=\hat t_1 < \hat t_2 < \dots < \hat t_{N}<\hat t_{N+1} = t_f$. Thus, we obtain a refinement of Proposition 4(a) of \cite{BH26}:
\begin{pro}\label{solutionoperatorproposition}
    The linear operator $\mathcal{S}_{d_1d_2} : L^2((0, t_f))^5\rightarrow C^0((0, t_f))^5$, defined by $\mu = \mathcal{S}_{d_1d_2}F_s$, where $\mu$ is the solution in $(0, t_f)$ of the linear inhomogeneous i.v.p.
    \begin{equation*}
        \dot \mu(t) = \frac{1}{t}L_{d_1d_2}\mu(t) + 2B_{d_1d_2}(\hat\eta(t), \mu(t)) + F_s(t), \qquad \mu(0)=0,
    \end{equation*}
    is bounded by
    \begin{equation*}
        M = \exp\bigg(\frac{t_f}{2} + 2\sum_{i=1}^N \|M_l(\hat\eta(\hat t_i))\|(\hat t_{i+1} - \hat t_i) + \|\dot M_l(\hat\eta(\hat t_i))\|(\hat t_{i+1} - \hat t_i)^2 + \frac{2}{3}C_2(\hat t_{i+1} - \hat t_i)^3 \bigg).
    \end{equation*}
\end{pro}
\begin{proof}
	Taking the dot product with $2\mu(t)$, we obtain the estimate
	\begin{align*}
		\frac{d}{dt}|\mu(t)|^2 &\leq 4|\mu(t)\cdot B_{d_1d_2}(\hat\eta(t), \mu(t))| + 2|\mu(t)||F_s(t)|\\
		&\leq (4\|M_l(\hat\eta(t))\|+1) |\mu(t)|^2 + |F_s(t)|^2,
	\end{align*}
	from which the lemma follows by Grönwall's inequality.
\end{proof}
Let $\hat E(t) \equiv \frac{1}{t}L_{d_1d_2}\hat\eta(t) + B_{d_1d_2}(\hat\eta(t), \hat\eta(t)) - \dot{\hat\eta}(t)$ and consider now the Einstein ODE
\begin{equation}\label{muODE}
    \dot\mu(t) = \frac{1}{t}L_{d_1d_2}\mu(t) + 2B_{d_1d_2}(\hat\eta(t), \mu(t)) + \hat E(t) + B_{d_1d_2}(\mu(t), \mu(t)), \qquad \mu(0)=0,
\end{equation}
which can be written as $\mu = \mathcal{S}_{d_1d_2}(\hat E + B_{d_1d_2}(\mu, \mu))$. Lemma A5 of \cite{BH26} states that $|B_{d_1d_2}(x,y)|\leq 3(d_1+1)|x||y|$. Thus, $\|\mathcal{S}_{d_1d_2}\hat E\|_{C^0((0, t_f))} \leq M \|\hat E(t)\|_{L^2((0, t_f))}$, as well as
\begin{equation*}
    \|\mathcal{S}_{d_1d_2}B_{d_1d_2}(\mu, \mu)\|_{C^0((0, t_f))}\leq M \cdot 3(d_1+1)\|\mu^2\|_{L^2((0, t_f))} \leq 3M(d_1+1)\sqrt{t_f}\|\mu\|_{C^0((0, t_f))}^2.
\end{equation*}
We can therefore use the Schauder fixed point theorem to show the equivalent of the $k=0$ part of Lemma 5 of \cite{BH26}.
\begin{lem}\label{schauderlemma}
    If an approximate solution $\hat\eta\in C^\infty((0, t_f))^5$ is chosen such that the \textup{a posteriori} error satisfies $\|\hat E(t)\|_{L^2((0, t_f))} < \frac{1}{12M^2(d_1+1)\sqrt{t_f}}$, then the Einstein ODE (\ref{muODE}) has a solution $\mu$ satisfying $\|\mu\|_{C^0((0, t_f))}\leq \frac{1}{6M(d_1+1)\sqrt{t_f}}$.
\end{lem}
We will construct such an approximate solution $\hat\eta$ via power series methods and Chebyshev interpolation in arbitrary-precision interval arithmetic, exactly as in \cite[\S 5]{BH26}.

\subsection{Critical point counting}\label{criticalcountingsection} We now count the critical points of $w$, of which there are finitely many\footnote{As they are non-degenerate, it suffices to show this near the $\alpha$ and $\omega$-limit sets: The $q_i^\pm$'s are sinks or sources, and near the smooth completions $p_i^\pm$ the boundary conditions impose a sign on $\dot w$.}. For the zero-counting process, we also need a good estimate on $\|\mu\|_{C^1((0, t_f))}$. First, we adapt Lemma A4 of \cite{BH26} to all dimensions $d_2\geq 2$.
\begin{lem}
    Consider the linear problem $\dot\mu(t) = \frac{1}{t}L_{d_1d_2}\mu(t) + F(t)$, with the initial data $\mu(0) = 0$. Then
    \begin{equation*}
        \|\mu\|_{C^1}\leq \left( 2\sqrt{\frac{d_2^2+8 + \sqrt{d_2^4 + 64}}{2}} + 2t_f + 1\right)\|F\|_{C^0}.
    \end{equation*}
\end{lem}
\begin{proof}
    We will only note the differences from the proof in \cite{BH26}. We have
    \begin{equation*}
        \exp(L_{d_1d_2}(\log(t) - \log(s))) = \begin{cases}
            \begin{pmatrix}
                1 & 0 & 0 & 0 & 0\\
                0 & (s/t)^2 & 0 & 0 & 0\\
                0 & 2(s/t)^2\log(t/s) & (s/t)^2 & 0 & 0\\
                0 & 0 & 0 & 1 & 0\\
                0 & 0 & 0 & 0 & 1
            \end{pmatrix} & \text{if $d_2 = 2$}\\
            \begin{pmatrix}
                1 & 0 & 0 & 0 & 0\\
                0 & (s/t)^{d_2} & 0 & 0 & 0\\
                0 & \frac{2}{d_2-2}\left(\frac{s^2}{t^2} - \left(\frac{s}{t}\right)^{d_2}\right) & (s/t)^2 & 0 & 0\\
                0 & 0 & 0 & 1 & 0\\
                0 & 0 & 0 & 0 & 1
            \end{pmatrix} & \text{otherwise.}\\
        \end{cases}
    \end{equation*}
For each $s<t$, we have that $\|\exp(L_{d_1d_2}(\log(t) - \log(s)))\|<2$. Then,
\begin{equation*}
    |\mu(t)|\leq \int_0^t 2|F(s)| ds \leq 2t_f \|F\|_{C^0}, \qquad \left|\frac{\mu(t)}{t}\right|\leq \frac{1}{t}\int_0^t 2|F(s) |ds \leq 2\|F\|_{C^0}.
\end{equation*}
Noting that
\begin{equation*}
    \|L_{d_1d_2}\| = \sqrt{\frac{d_2^2+8 + \sqrt{d_2^4 + 64}}{2}},
\end{equation*}
we get, using the ODE $\dot\mu(t) = \frac{1}{t}L_{d_1d_2}\mu(t) + F(t)$, the statement of the lemma.
\end{proof}

To get the desired $C^1$ estimate, we plug the Einstein equation $\mu = \mathcal{L}_{d_1d_2}(F_s + 2B_{d_1d_2}(\hat\eta,\mu))$ into the previous lemma, obtaining
\begin{equation*}
    \|\mu\|_{C^1} \leq \left( 2\sqrt{\frac{d_2^2+8 + \sqrt{d_2^4 + 64}}{2}} + 2t_f + 1\right)(1 + 6(d_1+1)M\sqrt{t_f}\|\hat\eta\|_{C^0})\|F_s\|_{C^0} \equiv M_1 \|F_s\|_{C^0}.
\end{equation*}
Thus, $\mathcal{S}_{d_1d_2}: C^0((0, t_f))\rightarrow C^1((0, t_f))$ has norm bounded above by $M_1$. Then,
\begin{align*}
	|\dot\mu(t)| \leq \|\mu\|_{C^1} &\leq M_1(\|\hat E\|_{C^0} + \|B_{d_1d_2}(\mu, \mu)\|_{C^0})\\
	&\leq M_1(\sqrt{t_f}\|\dot{\hat E}\|_{L^2}+|\hat E(0)| + 3(d_1+1)\|\mu\|_{C^0}^2),
\end{align*}
where we used a Sobolev inequality and the fact that $|B_{d_1d_2}(x,y)|\leq 3(d_1+1)|x||y|$.

Finally, we need to computationally identify the number of critical points of $w$, equivalently the number of zeroes of
\begin{equation*}
    \frac{\dot w}{w} = \frac{\dot f_1}{f_1} - \frac{\dot f_2}{f_2} = \left(\frac{1}{d_1} + \frac{1}{d_2}\right) \eta_2 - \frac{1}{d_2}\eta_3 - \frac{1}{t}.
\end{equation*}
It suffices therefore to count the zeroes of
\begin{equation*}
    p(t) \equiv \left(\left(\frac{1}{d_1} + \frac{1}{d_2}\right) \eta_2 - \frac{1}{d_2}\eta_3\right)t - 1.
\end{equation*}
To do so, we first write $p(t) = \hat p(t) + p^{(\mu)}(t)$, where $\hat p(t)=((\frac{1}{d_1} + \frac{1}{d_2}) \hat\eta_2 - \frac{1}{d_2}\hat\eta_3)t - 1$ and $p^{(\mu)}(t) = ((\frac{1}{d_1} + \frac{1}{d_2})\mu_2 - \frac{1}{d_2}\mu_3)t$, with $\mu_j$ the $j$-th component of $\mu$, is an error term satisfying $\|p^{(\mu)}\|_{C^0}\leq (\frac{2}{d_2} + \frac{1}{d_1})t_f\|\mu\|_{C^0}$ and $\|p^{(\mu)}\|_{C^1}\leq (\frac{2}{d_2} + \frac{1}{d_1})(t_f\|\mu\|_{C^1} + \|\mu\|_{C^0})$. Then, to reduce the computational burden in determining the zeroes with Sturm's theorem, we further decompose $\hat p$ into $\hat p(t) = \hat p^{(b)}(t) + \hat p^{(e)}(t)$, where $\hat p^{(b)}$ is the truncation of the Chebyshev polynomial down to some $K$ terms.

Consider two times $t_1<t_2$. We would like to show one of two statements: either $p(t)$ has no zeroes in $[t_1, t_2]$, or it has exactly one zero. The specific computational procedures used to do so are detailed below.

Let $\varepsilon_0>0$ be such that $\|\hat p^{(e)}\|_{C^0} + \|p^{(\mu)}\|_{C^0}\leq \varepsilon_0$. If there exists a $t^*\in[t_1,t_2]$ such that $p(t^*)=0$, then $|\hat p^{(b)}(t^*)|\leq \varepsilon_0$. Thus, if we have that $\hat p^{(b)}(t_1), \hat p^{(b)}(t_2)>\varepsilon_0$, it suffices to show that the low-degree polynomial $\hat p^{(b)} - \varepsilon_0$ has no zeroes in $(t_1, t_2]$ to conclude that $p$ has no zeroes in $[t_1, t_2]$, and similarly with the signs reversed.

Alternatively, there is at least one zero of $p$ in $(t_1, t_2)$ if $\hat p^{(b)}(t_1)>\varepsilon_0$ and $\hat p^{(b)}(t_2)<-\varepsilon_0$. To show that there are no other zeroes, let $\varepsilon_1>0$ be such that $\|\dot{\hat p}^{(e)}\|_{C^0} + \|p^{(\mu)}\|_{C^1}\leq \varepsilon_1$. If there were two zeroes, then there would be some $t^*\in (t_1, t_2)$ such that $\dot p(t^*)=0$ and $|\dot{\hat p}^{(b)}(t^*)|\leq \varepsilon_1$. Therefore, if $\hat p^{(b)}(t_1)>\varepsilon_0$, then we check that $\dot{\hat p}^{(b)}(t_1)<-\varepsilon_1$ and that $\dot{\hat p}^{(b)}+\varepsilon_1$ has no zeroes in $(t_1, t_2]$. The process is similar with the signs reversed.

All possible situations, with the corresponding (sufficient) criteria to check numerically using Sturm's theorem, are listed below.

\begin{itemize}
    \item No zeroes
    \begin{itemize}
        \item $p$ positive: $\hat p^{(b)}(t_1)>\varepsilon_0$ and $\hat p^{(b)} - \varepsilon_0$ has no zeroes in $(t_1, t_2]$.
        \item $p$ negative: $\hat p^{(b)}(t_1)<-\varepsilon_0$ and $\hat p^{(b)} + \varepsilon_0$ has no zeroes in $(t_1, t_2]$.
    \end{itemize}
    \item Exactly one zero
    \begin{itemize}
        \item $p$ decreasing: $\hat p^{(b)}(t_1)>\varepsilon_0$, $\hat p^{(b)}(t_2)<-\varepsilon_0$, $\dot{\hat p}^{(b)}(t_1)<-\varepsilon_1$ and $\dot{\hat p}^{(b)}+\varepsilon_1$ has no zeroes in $(t_1, t_2]$.
        \item $p$ increasing: $\hat p^{(b)}(t_1)<-\varepsilon_0$, $\hat p^{(b)}(t_2)>\varepsilon_0$, $\dot{\hat p}^{(b)}(t_1)>\varepsilon_1$ and $\dot{\hat p}^{(b)}-\varepsilon_1$ has no zeroes in $(t_1, t_2]$.
    \end{itemize}
\end{itemize}

\section{Standard and non-standard Einstein metrics}\label{theoremsection}

For each $d_1, d_2\geq 2$, there are two ``standard'' solutions coming out of $p_2^+$:
\begin{itemize}
    \item $f_1(t) = \cos(t)$, $f_2(t) = \sin(t)$ and thus $\alpha = \sqrt{d_1-1}$. This solution goes to $p_1^-$ and corresponds to the round metric on $S^{d_1+d_2+1}$.
    \item $f_1(t) = \sqrt{\frac{d_1-1}{n}}$, $f_2(t) = \sqrt{\frac{d_2}{n}}\sin\left(\sqrt{\frac{n}{d_2}}\,t\right)$ and thus $\alpha=\sqrt{n}$. This solution goes to $p_2^-$ and corresponds to the product of round metrics on $S^{d_1}\times S^{d_2+1}$.
\end{itemize}

\begin{proof}[Proof of Theorem \ref{maintheorem}]
    For a fixed choice of $(d_1, d_2)$, in order to use Lemma \ref{whichcompletion2}, we would like to find $\alpha, \beta>0$ such that $\eta_\alpha$ reaches $q_1^-$ and $\eta_\beta$ reaches $q_2^-$ with one fewer critical point of $w$ than $\eta_\alpha$.

    To do so, we first choose $t_f>0$, then construct an approximate solution $\hat\eta_\alpha$ satisfying the hypotheses of Lemma \ref{schauderlemma} using an arbitrary-precision power series solver and Chebyshev interpolation, as in \cite{BH26}, thus constructing an (incomplete) Einstein metric $\eta_\alpha(t)$ for $t\in(0,t_f)$ with an upper estimate $\epsilon$ on $|\eta_\alpha - \hat\eta_\alpha|$. These computations are implemented using the arbitrary-precision interval arithmetic package \texttt{arb} \cite{J17}; the code, with the values of all additional  technical parameters used to produce the metrics in this paper, can be found at \url{https://github.com/Qiu-Shi-Wang/einstein-spheres}. Some of the subroutines used are taken from the code provided with \cite{BH26} and \cite{W26}, with several improvements and optimisations made ``under the hood'' to significantly reduce the computational burden.
    
    Using (\ref{NWcoordinates}) and (\ref{BHcoordinates}), we obtain formulae $Z(\eta), \Delta(\eta), H(\eta)$ for the Nienhaus--Wink coordinates in terms of the Buttsworth--Hodgkinson coordinates. Since we know that $\eta_{\alpha,j}(t_f)\in [\hat\eta_{\alpha,j}(t_f)-\epsilon, \hat\eta_{\alpha,j}(t_f)+\epsilon]$ for each $j=1, 2, 3, 4, 5$, we can therefore check using interval arithmetic that $\left(Z(\eta_\alpha(t_f)), \Delta(\eta_\alpha(t_f)), H(\eta_\alpha(t_f))\right)$ satisfy the hypotheses of Lemma \ref{boxeslemma}, and thus $\eta_\alpha$ converges to $q_1^-$.

    As $|\Delta|'>0$ past $t=t_f$, the only critical points of $w$ for $\eta_\alpha$ are on $(0, t_f)$, which can be counted following the computational procedure in \S \ref{criticalcountingsection}.

    We use the same method to show that $\eta_\beta$ converges to $q_2^-$ and to count its number of critical points. If it is indeed one less than $\eta_\alpha$, then by Lemma \ref{whichcompletion2} there exists some $\gamma$ between $\alpha$ and $\beta$ such that $\eta_\gamma$ goes from $p_2^+$ to $p_1^-$ and thus corresponds to an Einstein metric on $S^{d_1+d_2+1}$, which is not round if $\alpha, \beta>\sqrt{d_1-1}$. Similarly, if $\eta_\beta$ has one more critical point than $\eta_\alpha$, then we have found an Einstein metric on $S^{d_1}\times S^{d_2+1}$, which is not the product of round metrics if $\alpha, \beta>\sqrt{n}$.

    We repeat this process for each metric listed in the theorem statement, with the values given in Table \ref{maintable}. There is an Einstein metric between each adjacent pair of rows of the table with the same $(d_1, d_2)$. Note that the metric between rows XIII and XIV is the one found in \cite{BH26}(with $\alpha\simeq 6.084$), while the metric between XIV and XV is new.
    
    For an explanation of why all of the metrics are distinct, see \cite[Remark 5.2]{NW25}.    
\end{proof}
\begin{table}[ht!]
        \centering
        \begin{tabular}{c|cc|c|c|c|cc|cc}
           & $d_1$ & $d_2$ & $\alpha$ & \#crit & $t_f$ & $\Delta_f$ & $H_f$ & sink & Einst. metr. \\ \hline
           I & 2 & 7 & 4 & 2 & $<2.553$ & $<-.561$ & $<-.980$ & $q_2^-$ & \multirow{2}{6em}{$\Bigr\}S^{10}$} \\
           II & 2 & 7 & 6 & 3 & $<2.736$ & $>.312$ & $<-.985$ & $q_1^-$ & \\ \hline
           III & 3 & 6 & 5 & 2 & $<2.524$ & $<-.400$ & $<-.962$ & $q_2^-$ & \multirow{2}{6em}{$\Bigr\}S^{10}$}\\
           IV & 3 & 6 & 10 & 3 & $<2.837$ & $>.252$ & $<-.979$ & $q_1^-$ & \\ \hline
           V & 4 & 5 & 5 & 2 & $<2.438$ & $<-.364$ & $<-.949$ & $q_2^-$ & \multirow{2}{6em}{$\Bigr\}S^{10}$}\\
           VI & 4 & 5 & 10 & 3 & $<2.813$ & $>.268$ & $<-.980$ & $q_1^-$ & \\ \hline
           VII & 7 & 2 & 7 & 3 & $<2.513$ & $>.515$ & $<-.985$ & $q_1^-$ & \multirow{2}{6em}{$\Bigr\}S^7\times S^3$}\\
           VIII & 7 & 2 & 9 & 4 & $<2.675$ & $<-.167$ & $<-.993$ & $q_2^-$ & \\ \hline \hline
           IX & 2 & 8 & 4.5 & 2 & $<2.556$ & $<-.532$ & $<-.977$ & $q_2^-$ & \multirow{2}{6em}{$\Bigr\}S^{11}$}\\
           X & 2 & 8 & 6.5 & 3 & $<2.721$ & $>.303$ & $<-.989$ & $q_1^-$ & \\ \hline
           XI & 3 & 7 & 8 & 2 & $<2.698$ & $<-.360$ & $<-.982$ & $q_2^-$ & \multirow{2}{6em}{$\Bigr\}S^{11}$} \\
           XII & 3 & 7 & 12 & 3 & $<2.872$ & $>.152$ & $<-.991$ & $q_1^-$ & \\ \hline \hline
           XIII & 2 & 9 & 5 & 2 & $<2.558$ & $<-.530$ & $<-.978$ & $q_2^-$ & \multirow{3}{6em}{$\Biggr\}S^{12}$ (twice)}\\
           XIV & 2 & 9 & 8 & 3 & $<2.734$ & $>.178$ & $<-.989$ & $q_1^-$ & \\
           XV & 2 & 9 & 50 & 2 & $<2.977$ & \footnotesize $<-.50139$ & \footnotesize $<-.99869$ & $q_2^-$ & \\ \hline \hline
           XVI & 2 & 10 & 5 & 2 & $<2.533$ & $<-.560$ & $<-.980$ & $q_2^-$ & \multirow{3}{6em}{$\Biggr\}S^{13}$ (twice)}\\
           XVII & 2 & 10 & 9 & 3 & $<2.720$ & $>.111$ & $<-.997$ & $q_1^-$ & \\
           XVIII & 2 & 10 & 25 & 2 & $<2.861$ & $<-.511$ & $<-.997$ & $q_2^-$
        \end{tabular}
        \caption{List of singular solutions starting from $p_2^+$ constructed to prove Theorem \ref{maintheorem}, and relevant associated quantities. \#crit is the number of critical points of $w$ and $\Delta_f\equiv\Delta(\eta_\alpha(t_f))$, $H_f\equiv H(\eta_\alpha(t_f))$.}
        \label{maintable}
    \end{table}
    
\textit{Remarks.} The ``double crossing'' behaviour observed in \cite[\S B.3]{BH26} appears to be exclusive to dimensions 12 and 13. In dimensions 14 and above, the curves seem to no longer cross, which would imply that there are no more non-round $SO(d_1+1)\times SO(d_2+1)$-invariant Einstein metrics on spheres.

All known positive Einstein metrics of the form (\ref{metric}) on $S^{d_1}\times S^{d_2+1}$ are ``symmetric'' in the sense of Böhm, i.e. $f_1(0) = f_1(T)$. It seems like this is also the case for our new metric on $S^7\times S^3$.

For illustrative purposes, plots of the warping functions $f_1, f_2$ for each sphere metric are shown in Figure \ref{plots}.

\begin{figure}[ht]
	\centering
	\begin{subfigure}{0.32\textwidth}
		\centering
		\includegraphics[width=1.1\linewidth]{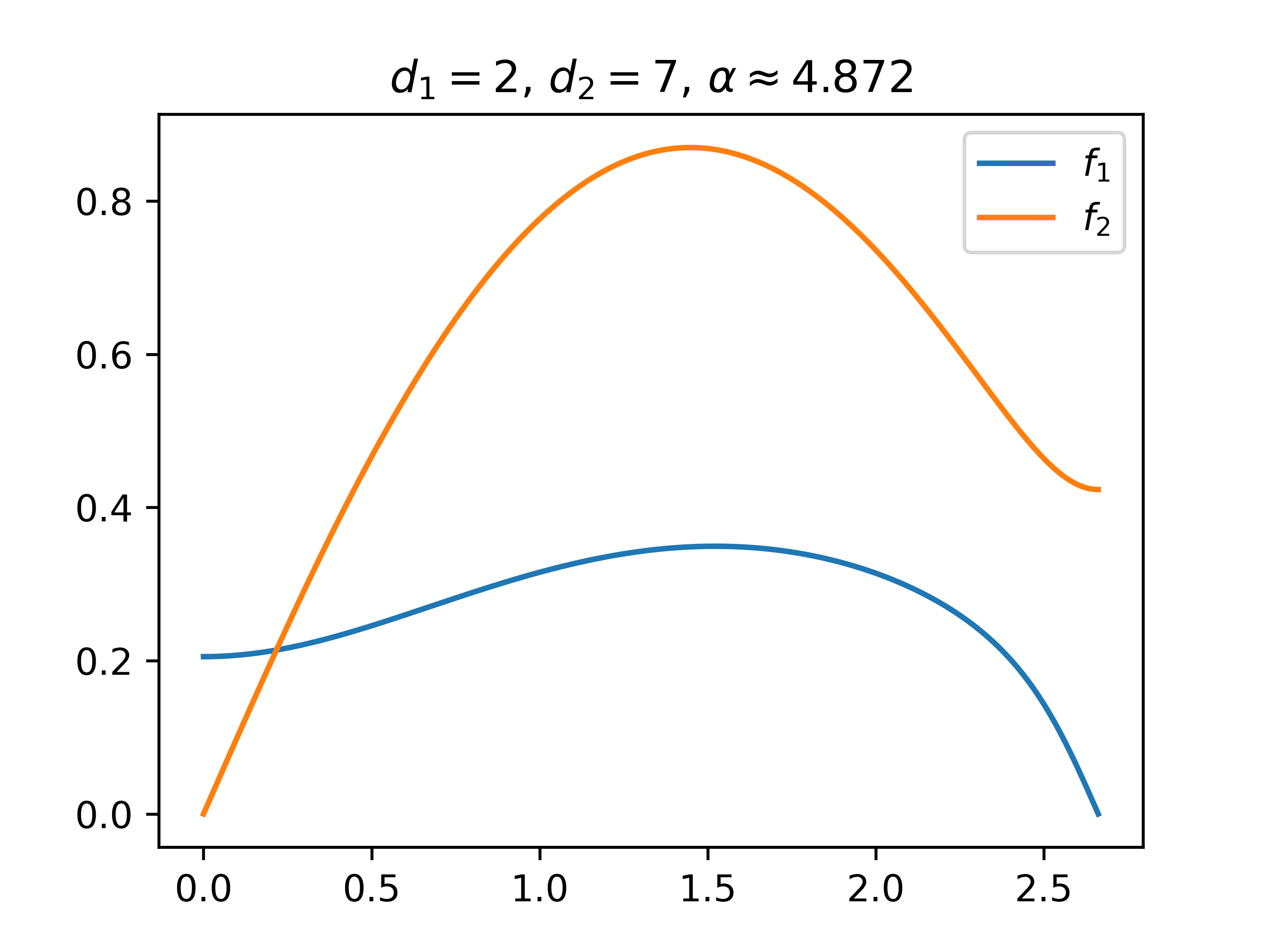}
	\end{subfigure}
	\begin{subfigure}{0.32\textwidth}
		\centering
		\includegraphics[width=1.1\linewidth]{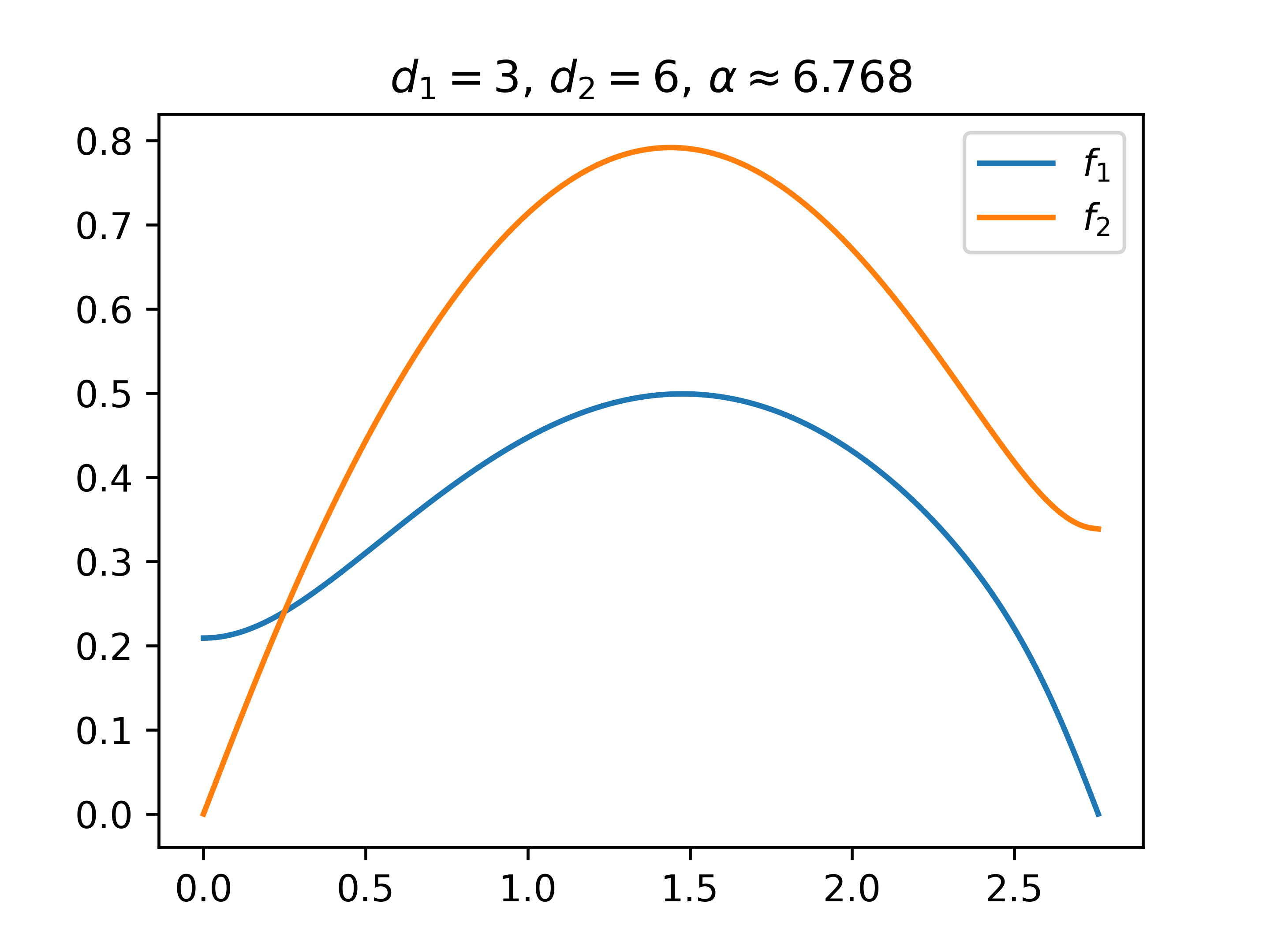}
	\end{subfigure}
	\begin{subfigure}{0.32\textwidth}
		\centering
		\includegraphics[width=1.1\linewidth]{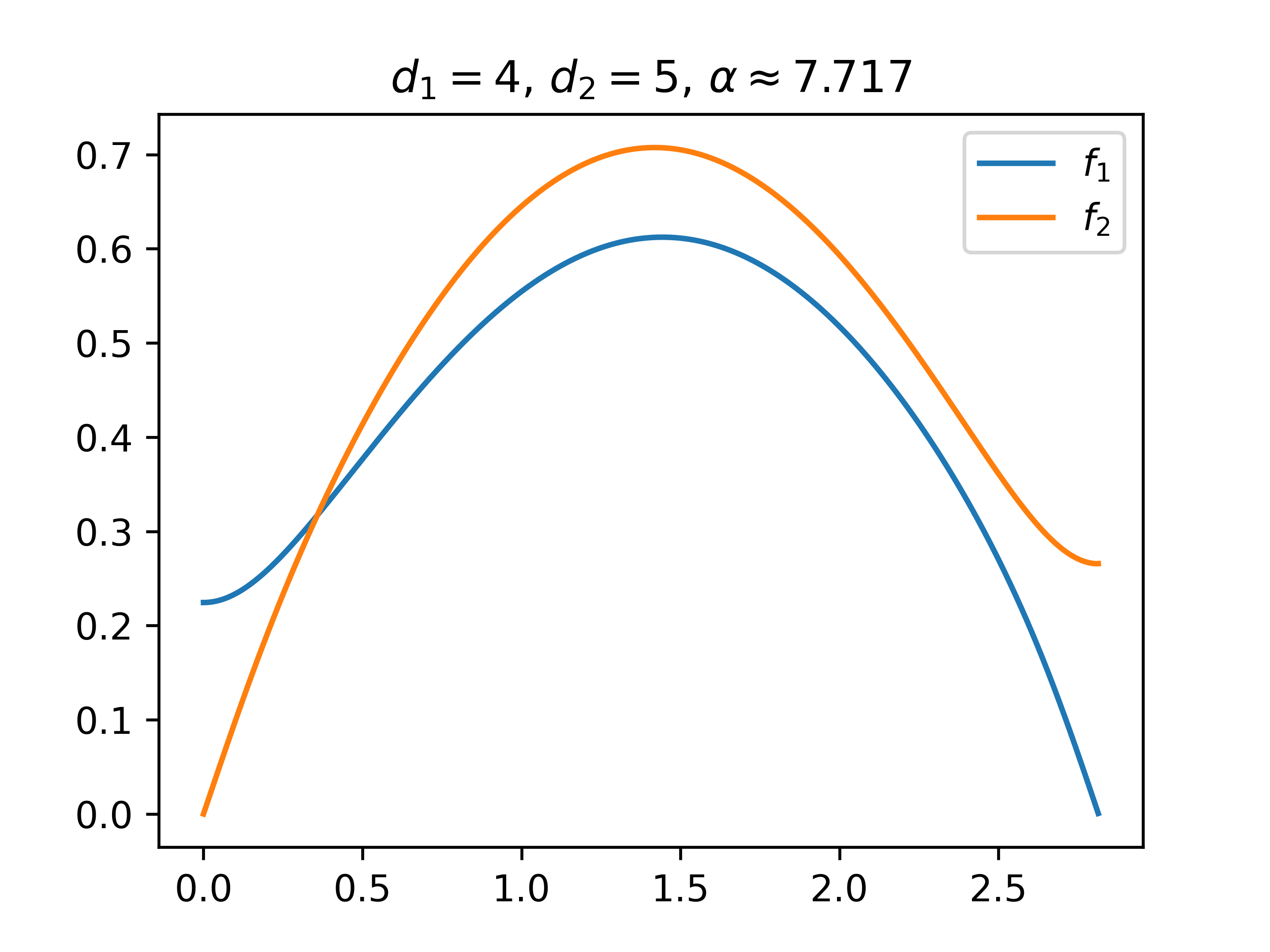}
	\end{subfigure}
	\begin{subfigure}{0.32\textwidth}
		\centering
		\includegraphics[width=1.1\linewidth]{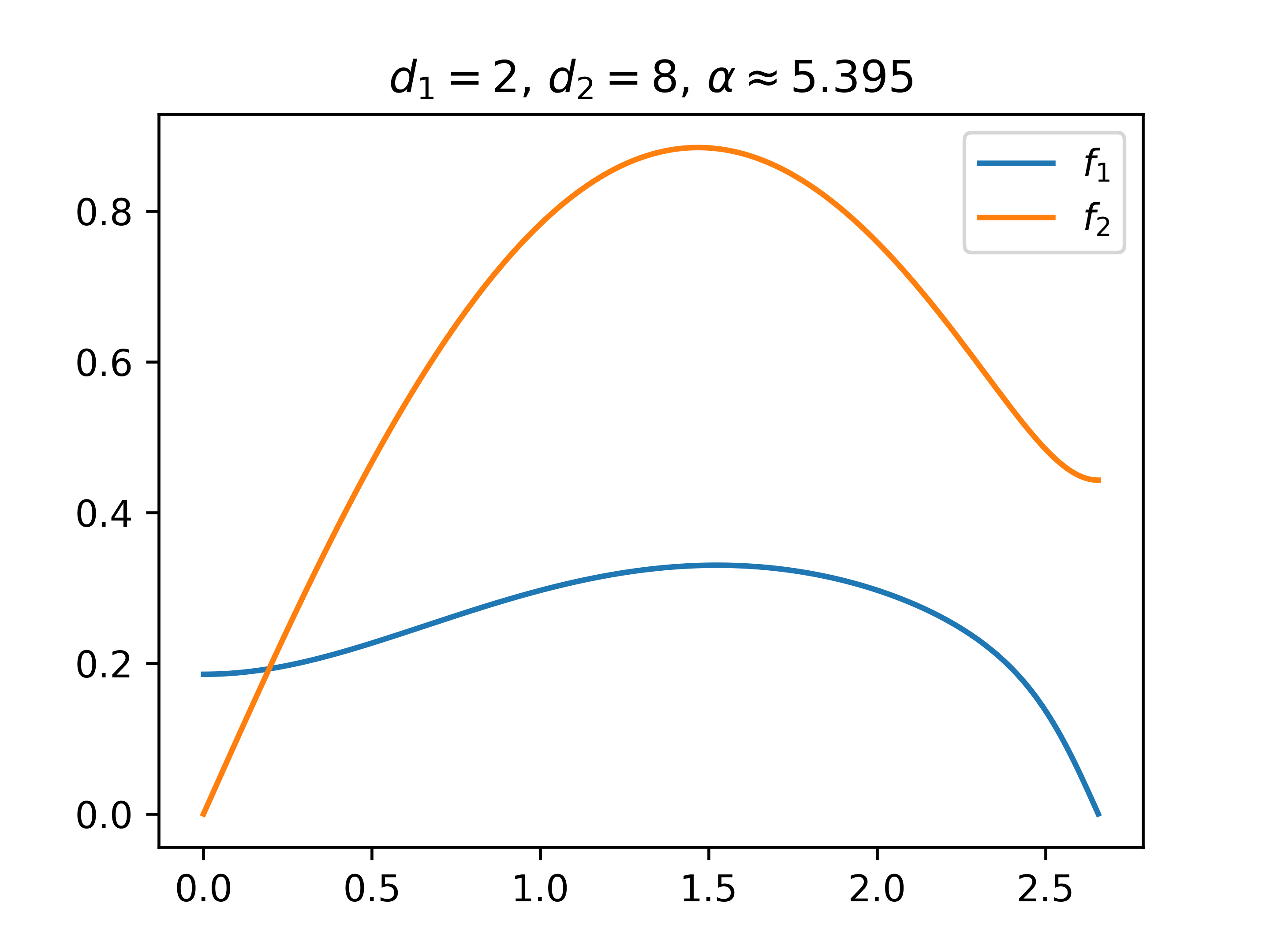}
	\end{subfigure}
	\begin{subfigure}{0.32\textwidth}
		\centering
		\includegraphics[width=1.1\linewidth]{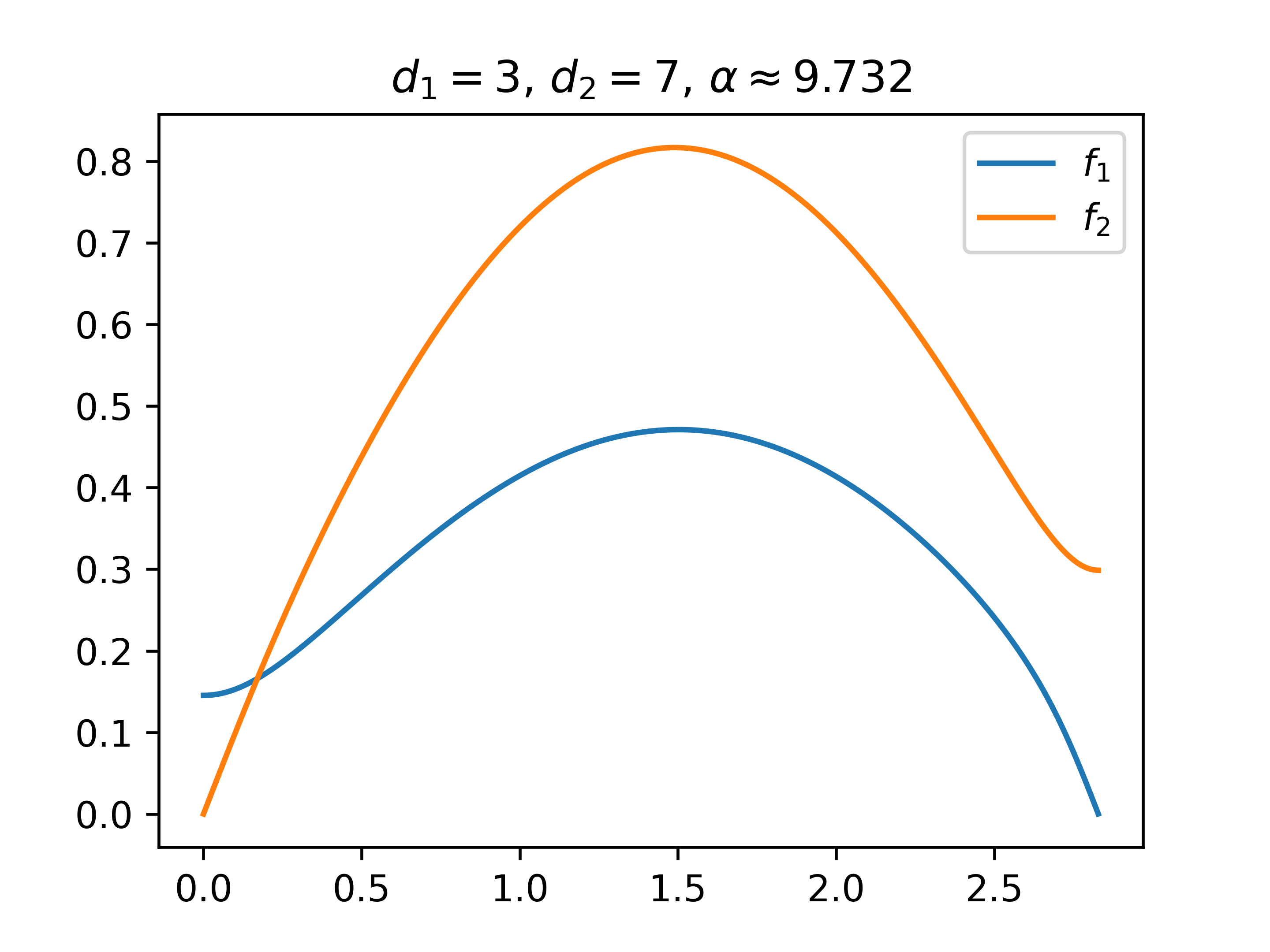}
	\end{subfigure}
	\begin{subfigure}{0.32\textwidth}
		\centering
		\includegraphics[width=1.1\linewidth]{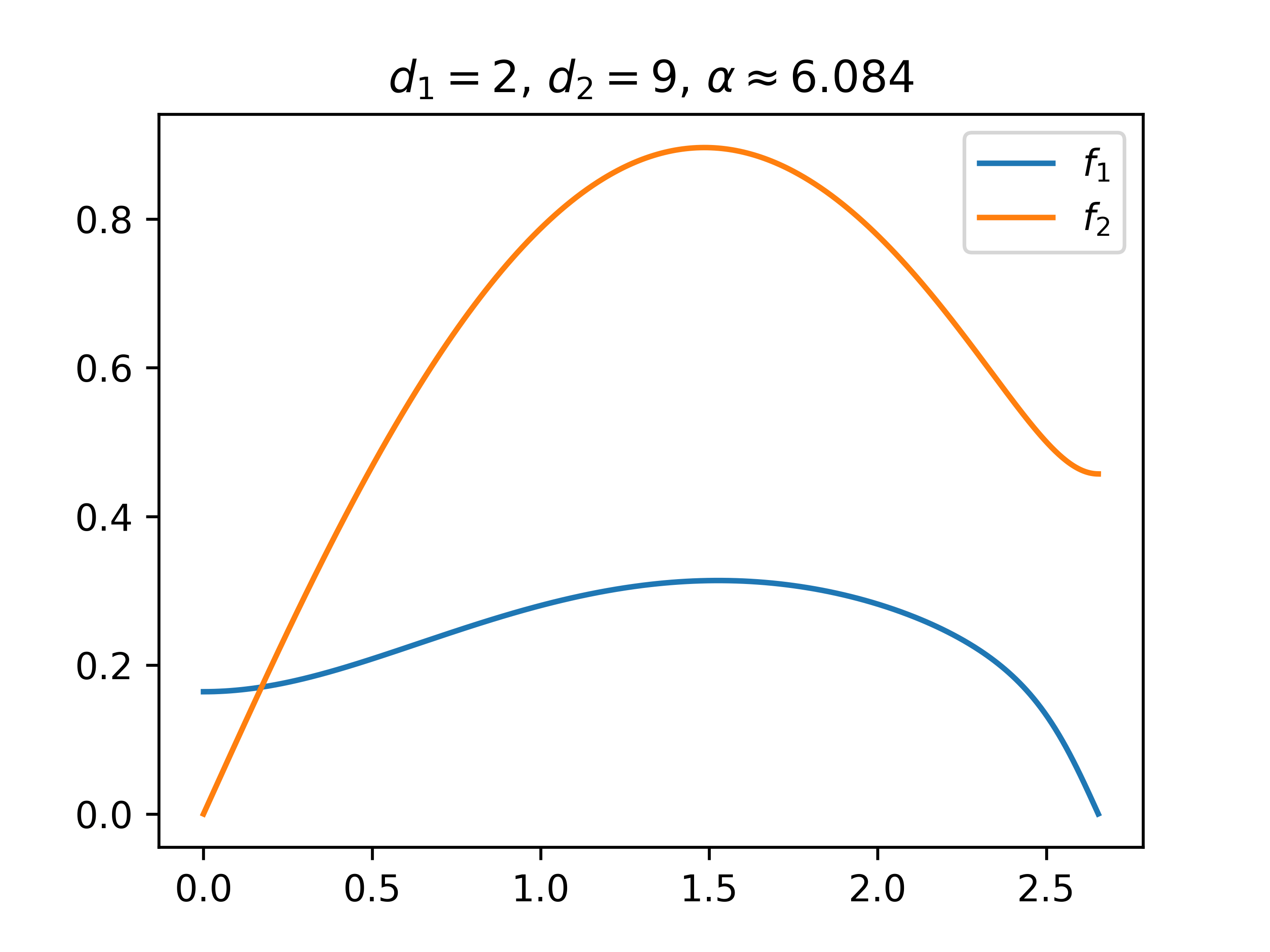}
	\end{subfigure}
	\begin{subfigure}{0.32\textwidth}
		\centering
		\includegraphics[width=1.1\linewidth]{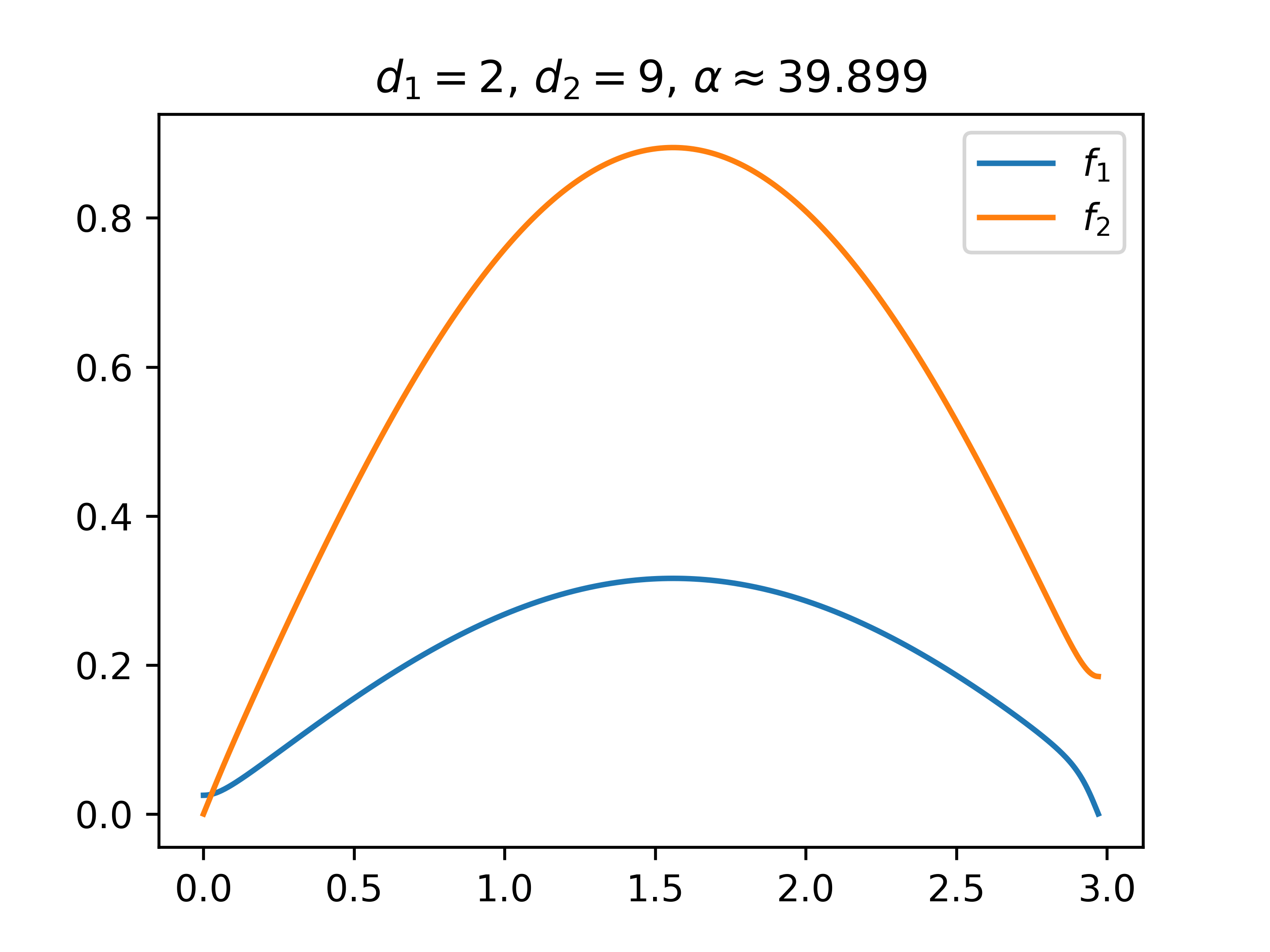}
	\end{subfigure}
	\begin{subfigure}{0.32\textwidth}
		\centering
		\includegraphics[width=1.1\linewidth]{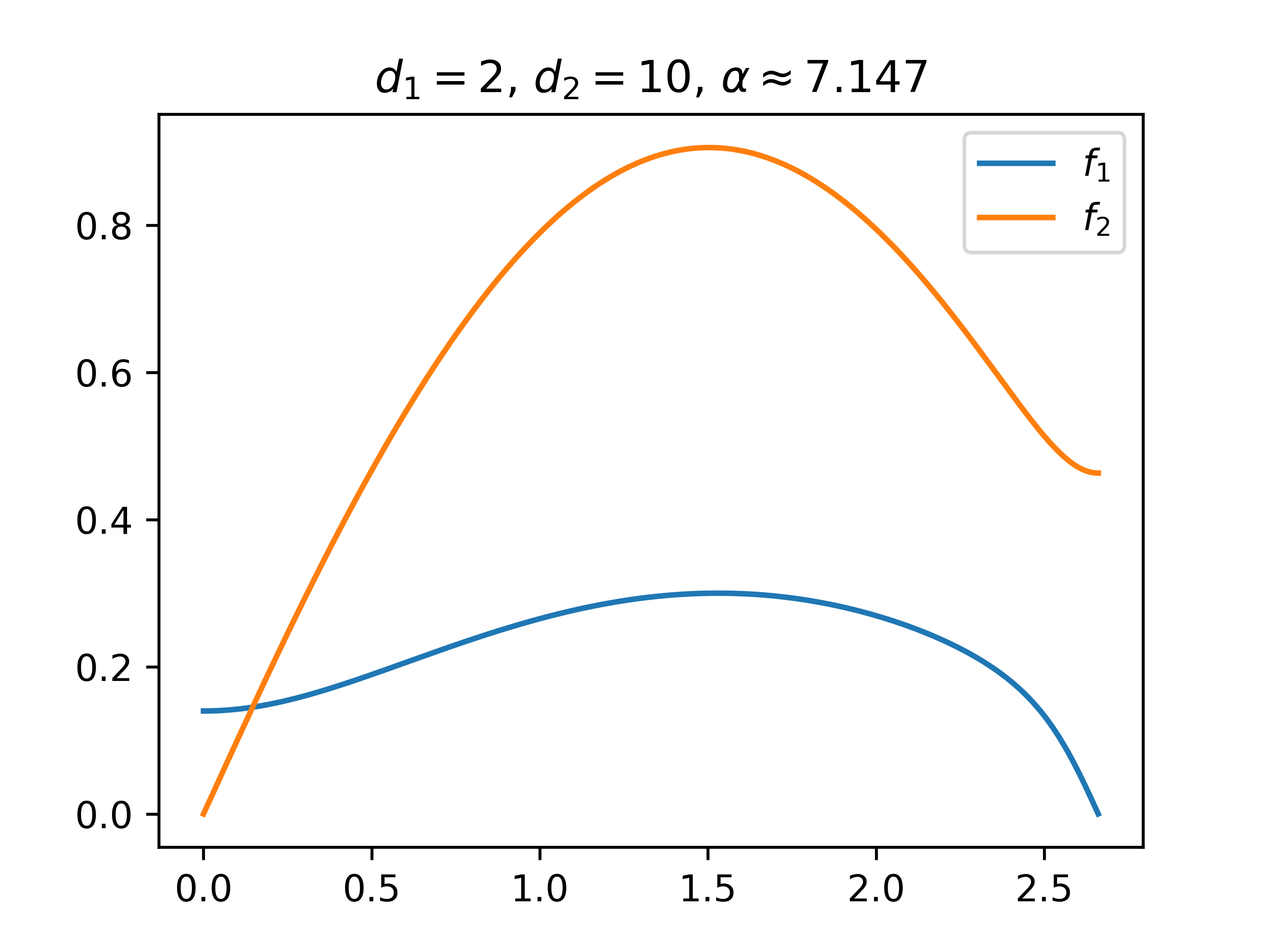}
	\end{subfigure}
	\begin{subfigure}{0.32\textwidth}
		\centering
		\includegraphics[width=1.1\linewidth]{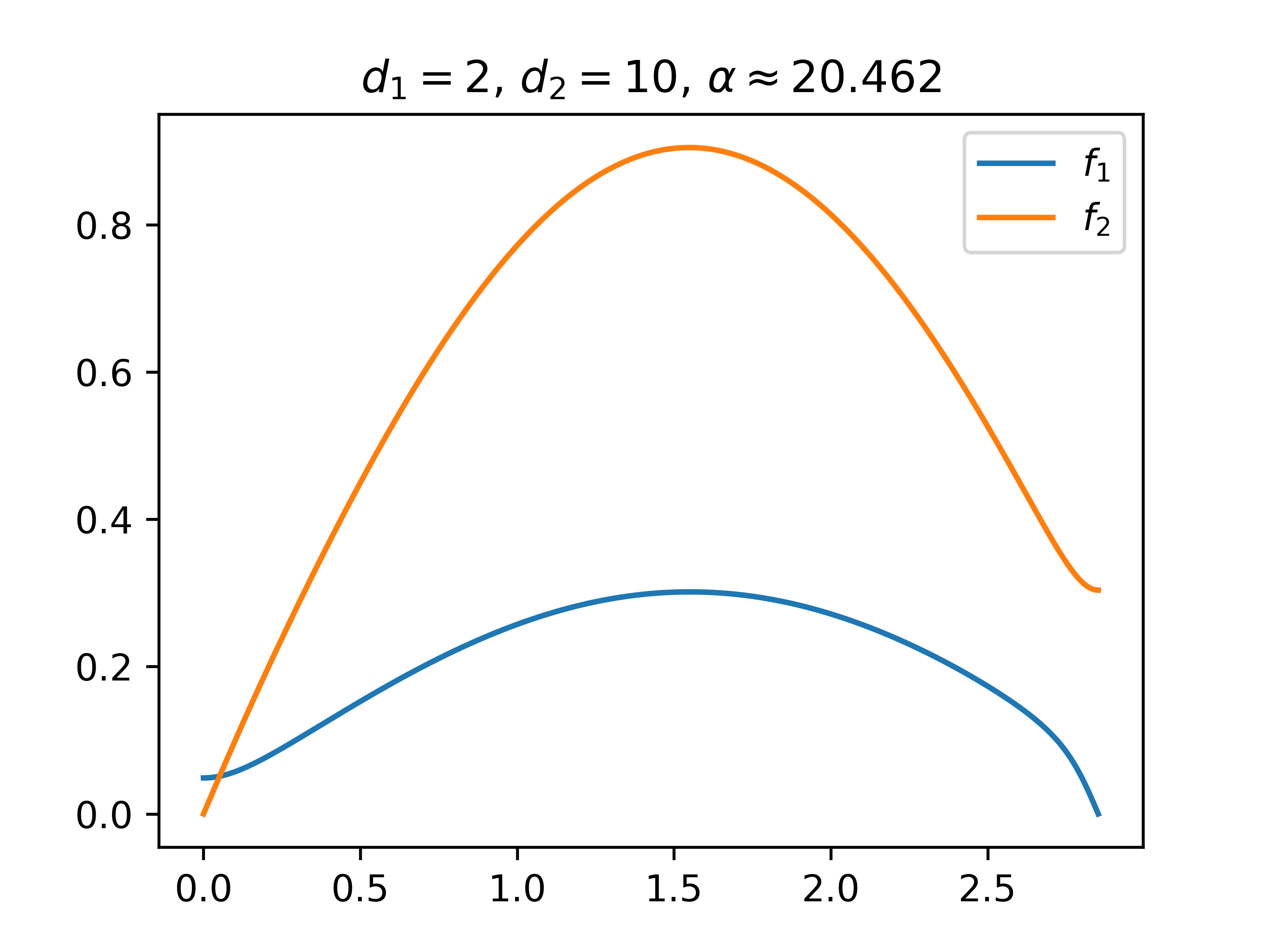}
	\end{subfigure}
	\caption{Plots of the warping functions $f_1(t), f_2(t)$ for each of the Einstein metrics on spheres of Theorem \ref{maintheorem}, with the respective approximate values of $\alpha=\sqrt{d_1-1}/f_1(0)$.}\label{plots}
\end{figure}

\medskip
\textbf{Acknowledgements.} I am grateful to Christoph Böhm for introducing to me the key ideas of this work during my visit to Münster in February 2026, as well as for his detailed comments and valuable guidance. I would also like to thank Andrew Dancer and Jason Lotay for their help, support and comments on the manuscript, as well as Tim Buttsworth and Matthias Wink for insightful discussions. Finally, my thanks go to the referee for their feedback and suggestions for improvement.

\bibliography{refs}
\bibliographystyle{amsplain}

\end{document}